\documentclass{article}

\usepackage{amsmath}
\usepackage{amssymb}
\usepackage{latexsym}
\usepackage{graphicx}
\usepackage{xypic}
\usepackage{amsfonts}
\usepackage{subfigure}
\usepackage{longtable}
\usepackage{url}
\usepackage{lscape}
\usepackage{hhline}
 \DeclareMathOperator \Ell {Ell}
 \DeclareMathOperator \PSU {PSU}
\DeclareMathOperator \PSL {PSL} \DeclareMathOperator \PGL {PGL}
\DeclareMathOperator \SL {SL} \DeclareMathOperator \GL {GL}
 \DeclareMathOperator \MCG {MCG}
\DeclareMathOperator \tr  {tr } \DeclareMathOperator\Hom{Hom}
\DeclareMathOperator \Homeo{Homeo} 
\DeclareMathOperator\Imaginary{Im} \DeclareMathOperator \Kernel{Ker}
\DeclareMathOperator \Aut {Aut} \DeclareMathOperator \Out {Out}

 \DeclareMathOperator \R {R}
 \DeclareMathOperator \Conj {conj}

\newtheorem{theorem}{Theorem}[section]
\newtheorem{prop}{Proposition}[section]

\newtheorem{lem}{Lemma}[section]
\newtheorem{defin}{Definition}[section]

\newenvironment{proof}[1][Proof]{\begin{trivlist}
\item[\hskip \labelsep {\bfseries #1}]}{\end{trivlist}}

\begin{document}


\title{The action of the mapping class group on representation
varieties of $\PSL(2,\mathbb{R})$. \\Case I: The one-holed torus.}

\author {Panagiota Konstantinou}

\maketitle
\begin{abstract}In this paper we consider the action of the mapping class group of a surface on
the space of homomorphisms from the fundamental group of a surface
into $\PSL(2,\mathbb{R})$. Goldman conjectured that when the surface
is closed and of genus bigger than one, the action on
non-Teichm\"uller connected components of the associated moduli
space (i.e. the space of homomorphisms modulo conjugation) is
ergodic. One approach to this question is to use sewing techniques
which requires that one considers the action on the level of
homomorphisms, and for surfaces with boundary. In this paper we
consider the case of the one-holed torus with boundary condition,
and we determine regions where the action is ergodic. Our main
result mirrors a theorem of Goldman's at the level of moduli.
\end{abstract}

\section{Introduction}\label{C:introduction} Let $G$ denote an
abstract Lie group which is isomorphic to $\PSL(2,\mathbb R)$. Let
$\Sigma$ denote the one-holed torus equipped with a basepoint and a
loop which connects the basepoint to the boundary  as in figure
\ref{fig:oneholetorus1}. The boundary component $c$  is associated
with a group element $g_c \in G$.  Given $g_c  \in G$, we write
$\Sigma_{g_c}$ to indicate that we impose the boundary condition
$g_c$. Let $\tilde{G}$ be the covering group of $G$. Since
$\pi_1(\Sigma_{g_c})$ is a free group on two generators, $\alpha, \
\beta$, we can identify
$$\Hom \bigl(\pi_1(\Sigma_{g_c}), G\bigl)=\{(g_{\alpha}, g_{\beta})
\in G \times G : [g_{\alpha}, g_{\beta}]=g_c \}.$$ Let  $R_1$ denote
the lift of the commutator mapping:
$$R_1: \ \begin{array}{rcccl}
\Hom\bigl(\pi_1(\Sigma_{g_c}),G\bigl) &\rightarrow &G \times
G&\rightarrow &\tilde{G} \\
 g &\mapsto& (g_{\alpha},
g_{\beta})&\mapsto&\tilde{g_c}
\end{array}$$
We define $\Gamma_{\Sigma}$ the subgroup of the mapping class group
that is generated by the two Dehn twists that do not affect the
basepoint and the loop that connects the basepoint to the boundary.
The purpose of this paper is to prove the following:

Let $\tilde{g_c} \in \tilde G$ a lift of $g_c$ to the universal
covering of $G$. If $g_c'\ \in \SL(2,\mathbb{R})$ the projection of
$\tilde{g_c}$ onto $\SL(2,\mathbb{R})$ and $t=\tr g_c'$, then

\begin{itemize}
\item  For $t<2$ the group $\Gamma_{\Sigma}$ acts properly discontinuously on $R_1^{-1}(\tilde{g}_c)$.
\item For $2<t<18$ the group $\Gamma_{\Sigma}$ acts ergodically on  $R_1^{-1}(\tilde{g}_c)$ for almost every $\tilde g_c$.
\item  For $t\geq 18$   there is an open subset  $\Omega_{\tilde{g}_c} \subset R_1^{-1}(\tilde{g}_c )$ such that the group $\Gamma_{\Sigma}$ acts properly discontinuously. On the complement of this subset, the action is ergodic for almost every $\tilde{g_c}$.
\end{itemize}

The result is an analogue to Goldman's result on the level of moduli
\cite{goldman03}. The proper discontinuity results follow directly
from the corresponding theorem of Goldman, \cite{goldman03}. Proving
ergodicity though becomes a harder question since Goldman is proving
ergodicity for an action on two dimensional spaces, whereas we
consider an action on three dimensional spaces. To prove ergodicity
we have to adjust the $L^2$ methods used in \cite{PX} for the
compact case in this noncompact settings. We use a combination of a
result of Goldman's (Theorem 5.2.1 in \cite{goldman03}) to reduce to
a pair of elliptic elements and the infinitesimal transitivity
method from \cite{PX}.

There are a lot of known results and also open questions on the
action of the mapping class group on representation varieties of
different groups (see Goldman's survey \cite{Goldman06}). Goldman
determined exactly the action of the mapping class group on the
moduli space of $\Hom(\pi_1\Sigma, \PSL(2,\mathbb{R}))$ in the case
where $\Sigma$ is the one-holed torus \cite{goldman03}. The next
step is to investigate the action of the mapping class group on
spaces of homomorphisms with $\Sigma$ a general compact surface. One
approach is to use the sewing techniques that were developed by
Pickrell and Xia in \cite{PX}. The idea is to obtain results for
surfaces with boundary starting with the one-holed torus and then
use the sewing lemma to obtain results for any surface. The sewing
lemma requires that one considers the space of homomorphisms and not
the moduli space. In addition the mapping class group that we
consider has to be restricted to a subgroup that does not effect the
basepoint or the loop that connects the basepoint to the boundary.
Given those restrictions we manage to prove that in the case of the
one-holed torus the result on the action of the mapping class group
on the space of homomorphisms does not differ from the action on the
corresponding moduli space.


 The structure of this paper is the following. In
section \ref{C:background} we give some background and the notation
that we will follow in the paper. In section
\ref{C:theoneholedtorus} we give the definition and the structure of
the one-holed torus. In section \ref{C:mappingclassgroup} we define
the mapping class group of the one-holed torus. For general facts on
one-holed tori, and the mapping class group we will refer to
Goldman's paper \cite{goldman03}. In section
\ref{C:actionofmappingclassgroup} we state the theorem and explain
the structure of the proof.  In section
\ref{C:infinitesimaltransitivity} we prove the infinitesimal
transitivity result which is crucial for the proof of ergodicity.
Finally in section \ref{C:proofoftheorem} we give the proof of the
theorem.
\subsection*{Acknowledgements}
I am grateful to my PhD thesis advisor, Doug Pickrell,  for
carefully reading this paper and giving many useful suggestions. I
also thank William Goldman for responding to number of questions
during this work.

\section{Background and notation}\label{C:background} Throughout this paper,
unless we state otherwise, $G$ will denote an abstract Lie group
which is isomorphic to $\PSL(2,\mathbb R)$. The elements of the
group $G$ fall into three classes: elliptic, parabolic, and
hyperbolic. If $g\in G$ is elliptic, it is useful to represent $G$
as the group $\PSU(1,1)$, the group of holomorphic automorphisms of
the unit disk, $\Delta\subset \mathbb C$; in this case $g$ is
conjugate to a rotation of the disk. Note that $|\tr g|<2 $. If
$g\in G$ is parabolic or hyperbolic, it is useful to represent $G$
as $\PSL(2,\mathbb R)$, the group of holomorphic automorphisms of
the upper half plane $\mathbb{H}^2\subset \mathbb C$; in this case
$g$ is conjugate to a translation or dilation. For parabolic
elements we have that $|\tr g|$=2 and for hyperbolic elements $|\tr
g|>2$.

  The fundamental group of $G$ is $\mathbb Z$, so that the universal covering map induces an exact sequence of groups
\begin{equation} \label{shortexactsequence}0\to \mathbb Z\to\tilde {G}\to G\to 0.\end{equation}
(For more details on the conjugacy classes of $G$ and its universal
covering see  chapter 2 in \cite{mydissertation})

Let $\Sigma$ denote a closed oriented surface with fixed basepoint.
Let $\gamma$ denote the genus of $\Sigma$. The space of
homomorphisms $\Hom\bigl(\pi_1(\Sigma) ,G\bigl)$ is called the
representation variety associated to $\Sigma$ and $G$. If we fix a
marking of $\Sigma$, i.e. a choice of standard generators of
$\pi_1(\Sigma)$, $\alpha_1$, $\beta_1$,..., $\alpha_{\gamma}$,
$\beta_{ \gamma}$, then we can identify $\Hom\bigl(\pi_1(\Sigma)
,G\bigl)$ with the set
$$\{(g_{\alpha_1},...,g_{\beta_{\gamma}})\in G^{2\gamma}: [g_{\alpha_1}, g_{\beta_1}]...[g_{\alpha_{\gamma}},g_{\beta_{\gamma}}]=1\}.$$
This is because the group $\pi_1(\Sigma)$ is defined by the single
relation $[\alpha_1,\beta_1]...[\alpha_{\gamma}, \beta_{\gamma}]=1$.
We let $H^1(\Sigma ,G)$ denote the space $\Hom\bigl(\pi_1(\Sigma)
,G\bigl)$ modulo the action of conjugation by $G$:
$$G\times \Hom \bigl(\pi_1(\Sigma) ,G\bigl)\to \Hom\bigl(\pi_1(\Sigma) ,G\bigl):(g,\phi )\to
\Conj(g)\circ\phi ,$$ where $\Conj(g)$ denotes the inner
automorphism of conjugation by $g$. This space does not depend upon
the choice of basepoint. Goldman and Hitchin have shown that the
representation variety $\Hom\bigl(\pi_1(\Sigma) ,G\bigl)$ consists
of finitely many connected components bounded in magnitude by
$|\chi(\Sigma)|$, where $\chi(\Sigma)$ is the Euler characteristic
of $\Sigma$. On a geometric point view, the connected component that
corresponds to the extreme value $|\chi(\Sigma)|$ is isomorphic to
the set of all possible ways of realizing $\Sigma$ as a quotient of
$\mathbb{H}^2$, and therefore it is all the possible universal
coverings with marking modulo isomorphism. This is the
Teichum\"uller space of $\Sigma$. (See chapter 4 in
\cite{mydissertation} for more extensive discussion on the component
that corresponds to the value $|\chi (\Sigma)|$). The variety
$\Hom\bigl(\pi_1(\Sigma),G\bigl)$ has a canonical
$\Gamma_{\Sigma}$-invariant measure class, the Lebesque class of the
set of nonsingular points.

Let $\MCG(\Sigma)=\pi_0\bigl(\Aut(\Sigma )\bigl)$ denote  the
mapping class group of $\Sigma$. This group acts naturally on
$\Hom\bigl(\pi_1(\Sigma) ,G\bigl)$; because $\Sigma$ is connected,
the mapping class group can be identified with the isotopy classes
(or homotopy classes) of homeomorphisms which fix our preferred
basepoint; with this understood, the action is given by
$$\MCG(\Sigma) \times \Hom\bigl(\pi_1(\Sigma) ,G\bigl)\to \Hom\bigl(\pi_1(\Sigma) ,G \bigl):([h],\phi )\to\phi\circ h_{*},$$
where $h_{*}$ is the automorphism of $\pi_1(\Sigma)$ induced by the
homeomorphism $h$ (which fixes our basepoint), and which does not
depend upon the choice of $h\in [h]$.

An interesting question is to understand the action of the mapping
class group on $\Hom\bigl(\pi_1(\Sigma) ,G\bigl)$.  Goldman
conjectured that when the surface is closed and of genus bigger than
one, the action on non-Teichm\"uller components of the associated
moduli space  is ergodic (conjecture 3.1 in \cite{Goldman06}). One
approach to this question is to use sewing techniques as in
\cite{PX}. The sewing method as developed by Pickrell and Xia (see
Sewing Lemma 1.3 p. 341) requires that one considers the action on
the level of homomorphisms and with surfaces with boundary and that
you start with the one-holed torus. In this paper we consider the
case of the one-holed torus with boundary condition, and we
determine the regions where the action is ergodic.

\section{The one-holed torus}\label{C:theoneholedtorus}
We consider a compact connected orientable surface of genus one with
one boundary component. Since attaching a disk to this surface
yields a torus, we refer to it as the {\it one-holed torus}.
\begin{figure}
 \[
\includegraphics {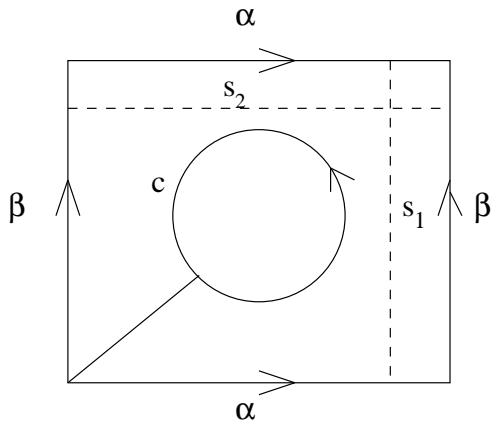}
\]
\caption{\label{fig:oneholetorus1} The one-holed torus, with group
element boundary condition}
\end{figure}

Let $\Sigma$ denote the one-holed torus, equipped with a basepoint
and a loop which  connects the basepoint to the boundary component
as in figure \ref{fig:oneholetorus1}. Given $g_c\in G$, we write
$\Sigma_{g_c}$ to indicate that we impose the boundary condition
$g_c$. Since $\Sigma$ can be continuously deformed to the figure
eight, given a choice of standard generators of $\pi_1(\Sigma)$,
$\alpha, \beta$ as in the figure \ref{fig:oneholetorus1}, $
\pi_1(\Sigma;x_0)$  admits the geometric presentation
$$\pi_1(\Sigma;x_0)=\langle \alpha, \beta, c : [\alpha,\beta]=c
\rangle$$ where $c$ corresponds to the generator of $\pi_1 (
\partial \Sigma)$. Then we can identify
$$\Hom \bigl(\pi_1(\Sigma_{g_c}), G\bigl)=\{(g_{\alpha}, g_{\beta})
\in G \times G : [g_{\alpha}, g_{\beta}]=g_c \}.$$ We define the
lifted commutator mapping  $$R_1 \ :\ G\times G\rightarrow \tilde{G}
\ :\ (g,h)\rightarrow [\tilde{g},\tilde{h}].$$ In terms of  $R_1$,

$$R_1: \ \begin{array}{rcccl}
\Hom\bigl(\pi_1(\Sigma_{g_c}),G\bigl) &\rightarrow &G \times G&\rightarrow &\tilde{G} \\
g &\mapsto& (g_{\alpha}, g_{\beta})&\mapsto&\tilde{g_c}
\end{array}$$
we have the decomposition into connected components
\begin{equation}\Hom\bigl(\pi_1(\Sigma_{g_c}),G\bigl)=\bigsqcup_{\tilde {g}_c } R_1^{-1}(\tilde {g_c})\end{equation}
where $\tilde {g}_c\in\tilde {G}$ is a lift of $g_c\in G$.

\section{The mapping class  group of the one-holed
torus}\label{C:mappingclassgroup}

Since $\pi_1(\Sigma;x_0)$ is freely generated by $\alpha$ and
$\beta$, the first homology group $H_1 (\Sigma, \mathbb{Z})$ is
isomorphic to $\mathbb{Z} \alpha \oplus \mathbb{Z} \beta$. The
action on the homology $H_1 (\Sigma, \mathbb{Z})$ defines a
homomorphism
$$h : \Out \bigl(\pi_1(\Sigma;x_0)\bigl) \longrightarrow \GL(2, \mathbb{Z}).$$
 Nielsen proved that the map $h$ defined above  is an isomorphism in the case of the one-holed torus \cite{nielsen}.  (See \cite{LS}, Proposition 4.5 or Magnus-Karrass-Solitar \cite{MKS}, Section 3.5, Corollary N4). This property does not generalize to other hyperbolic surfaces with boundary.

In this paper we are going to consider only a subgroup of the
$\MCG(\Sigma)$, that fixes the basepoint and the loop which joins
the basepoint and the boundary of $\Sigma$. This is because we would
like to use sewing techniques to generalized our results to higher
genus surfaces. We define two elements of the $\MCG(\Sigma)$.
 The {\it Dehn Twist} about $\alpha$ is the automorphism $\tau_{\alpha} \in \Aut \bigl(\pi_1(\Sigma;x_0)\bigl)$:
\begin{eqnarray}
\alpha &\rightarrow &\alpha\nonumber\\
\beta &\rightarrow & \beta \alpha\nonumber
\end{eqnarray}
and it corresponds to $\begin{pmatrix}
 1&0 \\
 1&1
 \end{pmatrix}$ $\in \SL(2, \mathbb{Z})$.

 The {\it Dehn Twist} about $\beta$ is the automorphism $\tau_{\beta} \in\Aut \bigl(\pi_1(\Sigma;x_0)\bigl)$:
\begin{eqnarray}
\alpha &\rightarrow &\alpha \beta\nonumber\\
\beta &\rightarrow & \beta\nonumber
\end{eqnarray} and it corresponds to $\begin{pmatrix}
 1&1 \\
 0&1
 \end{pmatrix}$ $\in \SL(2, \mathbb{Z})$.
The two Dehn twists generate a subgroup of $\MCG(\Sigma)$ isomorphic
to $\SL(2, \mathbb{Z})$ (See \cite{rankin} page 11).

As the $\MCG(\Sigma)$ acts on $\pi_1(\Sigma)$ , it would also act on
the space of $\Hom( \pi_1\Sigma,G)$ and as a consequence on $G\times
G$. Define $\Gamma_{\Sigma}$ to be the group generated by the
transformations $T_j:G\times G\to G\times G$ given by
\begin{equation}\label{twists}T_1(g,h)=(gh^{-1},h),\quad T_2(g,h)=(g,hg^{-1}).\end{equation}
These transformations arise from twists along the curves $s_1$ and
$s_2$ indicated in the  figure \ref{fig:oneholetorus1} as defined
above. They are volume-preserving (with respect to Haar measure),
they commute with conjugation by $G$, and they commute with the map
$R_1$.  The action of $\Gamma_{\Sigma}$ (the orientation preserving
mapping class group of $\Sigma$) restricts to the action of
$\pi_0\bigl(\Aut(\Sigma_{g_c})\bigl)$   on $R_1^{-1}(\tilde
{g}_c)\subset \Hom\bigl(\pi_1(\Sigma_{g_c}),G\bigl)$, for each
boundary condition. (For general facts about the mapping class group
see section 1 in \cite{goldman03}.)

\section{The action of the mapping class
group}\label{C:actionofmappingclassgroup}

\begin{theorem}\label{actionofMCG}
Let $\Sigma$ be the one-holed torus with additional structure
defined as in section \ref{C:theoneholedtorus}, and
$\Gamma_{\Sigma}$ the group generated by the two Dehn twists $T_1$
and $T_2$. Let $g_c$ be the group element that corresponds to the
boundary component,  and let $\tilde {g_c} \in \tilde{G}$ a lift of
$g_c$ to the universal covering of $G$. Let $g_c' \in
\SL(2,\mathbb{R})$ the projection of $\tilde{g_c}$ onto
$\SL(2,\mathbb{R})$ and $t= \tr g_c'$. Then
\begin{itemize}
\item  For $t<2$ the group $\Gamma_{\Sigma}$ acts properly discontinuously on $R_1^{-1}(\tilde{g}_c)$.
\item For $2<t<18$ the group $\Gamma_{\Sigma}$ acts ergodically on  $R_1^{-1}(\tilde{g}_c)$ for almost every $\tilde g_c$.
\item  For $t\geq 18$   there is an open subset  $\Omega_{\tilde{g}_c} \subset R_1^{-1}(\tilde{g}_c )$ such that the group $\Gamma_{\Sigma}$ acts properly discontinuously. On the complement of this subset, the action is ergodic for almost every $\tilde{g_c}$.
\end{itemize}

\end{theorem}

The open subset $\Omega_{\tilde{g}_c}$ is described in
\cite{goldman03} (section 5)  indirectly by giving a fundamental
domain, (also see section \ref{C:proofoftheorem}).  For the proof of
the ergodic part, we are going to adjust the  method that was
discovered in \cite{PX} for proving the ergodicity of mapping class
group actions on $\Hom \bigl(\pi_1 (\Sigma), K\bigl )$, where $K$ is
a compact group. The basic idea is that $\Gamma_{\Sigma}$-ergodicity
for almost every boundary condition, is equivalent to a question of
$\mathcal{G}$-ergodicity on orbits, where $\mathcal{G}$ is a {\it
continuous} group of volume-preserving transformations. This
$\mathcal{G}$ transitivity is locally reducible to a question about
infinitesimal transitivity (see section
\ref{C:infinitesimaltransitivity}). This method though, can be used
only on elliptic elements, so we modify Theorem 5.2.1 in
\cite{goldman03}   to show that in certain cases we can act by
$\Gamma_{\Sigma}$ and $\mathcal{G}$ and change pairs of elements in
$G\times G$ to pairs of elliptic elements (Lemma \ref{lemma2<t<18})
and we  obtain global transitivity. For the proper  discontinuity we
base our results on Goldman's theorem \cite{goldman03}. For the
convenience of the reader we recall his setting.


 He parameterizes the space $\Hom \bigl( \pi_1(\Sigma), \SL(2,\mathbb{C})\bigl)//\SL(2,\mathbb{C})$ by the traces $x, y, $ and $z$ of the generators $g$, $h$ and $g h$ respectively. In terms of these coordinates, the trace of the commutator $[g,h]$ is given by the polynomial $$\tr[g, h]= \kappa (x,y,z)=x^2+y^2+z^2-xyz-2.$$
Recall the definition of the map $\chi$.
$$\begin{array}{lrll}\chi :& \Hom\bigl(\pi_1(\Sigma), \SL(2,\mathbb{C})\bigl)//\SL(2,\mathbb{C})& \longrightarrow& \mathbb{C}^3\\ & (g,h)& \mapsto& (x,y,z)=(\tr g,\tr h, \tr gh).\end{array}$$
$\chi$ is an equivariant map:
$$\pi_0\bigl(\Homeo(\Sigma)\bigl)\times \Hom\bigl(\pi_1(\Sigma),
\SL(2,\mathbb{C})\bigl)// \SL(2,\mathbb{C}) \longrightarrow
\Aut(\kappa)\times \mathbb{C}^3.$$ Notice that the map
$\tilde{\chi}$ below has kernel the subgroup generated by the
elliptic involution $E=\left(\begin{smallmatrix}
-1&0\\0&-1\end{smallmatrix}\right)$: $$ 0\longrightarrow
\mathbb{Z}_2E\longrightarrow
\pi_0(\Homeo(\Sigma))\xrightarrow{\tilde{\chi}} \Aut(\kappa)$$ and
therefore
$\tilde{\chi}\left(\pi_0\bigl(\Homeo(\Sigma)\bigl)\right)\cong\PGL(2,\mathbb{Z})$
and  $\tilde{\chi}(\Gamma_{\Sigma})\cong\PSL(2,\mathbb{Z})$.

In this setting Goldman proves the main theorem in \cite{goldman03}
about the action of the group of automorphisms of the polynomial
$\kappa$, $\Aut (\kappa)$ on the space
$$\kappa^{-1}(t)\cap \mathbb{R}^3=\{(x,y,z) \in \mathbb{R}^3 \text{
such that } \kappa(x,y,z)=t \}.$$

In the cases where the action of $\Aut(\kappa)$ on $k^{-1}(t)\cap
\mathbb{R}^3$ in Goldman's setting is properly discontinuous, it
follows directly that the action of $\Gamma_{\Sigma}$   on
$\R_1^{-1}(\tilde{g}_c)$ is properly discontinuous. (Lemma
\ref{properaction}).

\begin{lem}\label{properaction}  Fix $\tilde g_c \in \tilde{G}$ that projects to  $g_c' \in \SL (2,\mathbb{R})$ and let $t=\tr (g_c')$.
 If $\Aut(\kappa)$ acts properly discontinuously on   $\kappa^{-1}(t)\cap \mathbb{R}^3$, then  $\Gamma_{\Sigma}$ acts properly discontinuously on  $R_1^{-1}(\tilde{g_c})$.

\end{lem}

\begin{proof} If $\Aut (\kappa)$ acts properly discontinuously on $\kappa^{-1}(t) \cap \mathbb{R}^3$ then $\pi_0\bigl(\Homeo (\Sigma)\bigl)$ acts properly discontinuously
on $\Hom \bigl(\pi_1( \Sigma),
\SL(2,\mathbb{R})\bigl)//\SL(2,\mathbb{R})$ and therefore away from
the parabolic elements it would act properly discontinuously on
$\Hom \bigl(\pi_1( \Sigma),
\SL(2,\mathbb{R})\bigl)/\SL(2,\mathbb{R})$ as well. Using the
quotient map $q$ $$q: \Hom \bigl(\pi_1( \Sigma),
\SL(2,\mathbb{R})\bigl) \longrightarrow \Hom \bigl(\pi_1( \Sigma),
\SL(2,\mathbb{R})\bigl)/\SL(2,\mathbb{R}).$$ we obtain that
$\Gamma_{\Sigma}$ acts properly discontinuous on
$\Hom\bigl(\pi_1(\Sigma),\SL(2,\mathbb{R})\bigl)$.

\hfill $\square$\end{proof}

Lemma \ref{lemma2<t<18}, is a modification of theorem 5.2.1 in
\cite{goldman03}. Theorem 5.2.1 in \cite{goldman03} says (among
other things) that if $\kappa(x,y,z)>2$, there exists  $\gamma \in
\Aut(\kappa)$ such that one of the following holds:
\begin{enumerate}
\item $\gamma \cdot (x,y,z) \in (-2,2) \times \mathbb{R}\times \mathbb{R}$
\item $ \gamma \cdot (x,y,z)\in (-\infty,-2)\times (-\infty,-2)\times(-\infty,-2)$
\end{enumerate}
\begin{lem} \label{lemma2<t<18}  Let $\tilde g_c \in \widetilde{G}$  where $\tilde {g_c}$ projects to $g_c' \in \SL (2,\mathbb{R})$ and suppose  $t =\tr (g_c') $. Consider  the set   $$\{ (g,h) \in \SL(2,\mathbb{R})\times \SL(2, \mathbb{R})  : [g,h]=g_c'\}$$ and let $x=\tr g\ , y=\tr h $ and $z=\tr g h$. The coordinates $x, y, $ and $z$ satisfy  $\kappa(x,y,z):=x^2+y^2+z^2-xyz-2=t$.  Suppose there exists $\gamma \in \Aut(\kappa)$ such that $\gamma \cdot (x,y,z) \in (-2,2)\times \mathbb{R} \times \mathbb{R}$, then  there exists $\gamma' \in \Gamma_{\Sigma}$ such that $\gamma' ( g, h)= (g',h')$ where $g'$ is an elliptic element.
\end{lem}

The idea of the proof of  lemma \ref{lemma2<t<18} is assuming
theorem 5.2.1 in \cite{goldman03}    make a similar statement
firstly about the action of $\pi_0\bigl(\Homeo(\Sigma)\bigl)$ on the
moduli space
$\Hom\bigl(\pi_1(\Sigma,\SL(2,\mathbb{R})\bigl)/\SL(2,\mathbb{R})$
and finally to prove the same statement for the group
$\Gamma_{\Sigma}$ (the group generated by the two Dehn twists $T_1$
and $T_2$ on $\Hom\bigl(\pi_1(\Sigma),G\bigl)$.  The  groups $\Aut
(\kappa)$, $\pi_0(\Homeo(\Sigma))$, and $\Gamma_{\Sigma}$ are
related, so  before we start the proof of the lemma, we are going to
understand the relation between these groups.


 From section 2.2 in \cite{goldman03} we have   that \begin{equation}\label{Autofkappa}\Aut (\kappa)=\PGL(2,\mathbb{Z})\ltimes \bigl(\mathbb{Z}_2\sigma_1 \oplus \mathbb{Z}_2\sigma_2\bigl)\end{equation} where $\PGL(2,\mathbb{Z})$ is the quotient of $\GL(2,\mathbb{Z})$ by $\{ \pm I\}$  and $\GL(2,\mathbb{Z}) \cong \pi_0\bigl(\Homeo (\Sigma)\bigl)$. $\mathbb{Z}\sigma_1\oplus \mathbb{Z}\sigma_2$ is the group of  sign-change automorphisms $\sigma_1(x,y,z)=(x,-y,-z)$, and $\sigma_2(x,y,z)=(-x,y,-z)$.
On the other hand we have the following exact sequences
$$\begin{array}{ccccccccc}
0 & \longrightarrow & \Gamma_{\Sigma} & \hookrightarrow & \pi_0(\Homeo(\Sigma)) &\xrightarrow{\det} & \mathbb{Z}_2 &\longrightarrow & 0\\
&&\downarrow{\cong}  &&\downarrow{\cong}&&&\\
0 & \longrightarrow &\SL(\mathbb{Z}_{\alpha}\oplus\mathbb{Z}_{\beta})&\hookrightarrow & \GL(\mathbb{Z}_{\alpha}\oplus \mathbb{Z}_{\beta}) & \xrightarrow{\det}&\mathbb{Z}_2 &  \longrightarrow & 0\\
&&\downarrow &&\downarrow&&&&\\
0 & \longrightarrow &\PSL(\mathbb{Z}_{\alpha}\oplus\mathbb{Z}_{\beta})&\hookrightarrow & \PGL(\mathbb{Z}_{\alpha}\oplus \mathbb{Z}_{\beta}) & \xrightarrow{\det}&\mathbb{Z}_2 &  \longrightarrow & 0\\

\end{array}$$

The later  sequence splits and therefore we have that
$$\PGL(2,\mathbb{Z})=\PSL(2,\mathbb{Z}) \rtimes \mathbb{Z}_2
\left(\begin{smallmatrix} -1&0\\0&1\end{smallmatrix}\right)$$ or we
could write $$\pi_0\bigl(\Homeo(\Sigma)\bigl)=\Gamma_{\Sigma}
\rtimes \mathbb{Z}_2 \left(\begin{smallmatrix}
-1&0\\0&1\end{smallmatrix}\right)$$ The element
$\left(\begin{smallmatrix} -1&0\\0&1 \end{smallmatrix}\right)$
corresponds to a  reflection $Q : (g,h)\rightarrow(g^{-1},h)$.

We now start the proof of lemma \ref{lemma2<t<18}

\begin{proof}

We observe that the elements of $\mathbb{Z}_2\sigma_1\oplus
\mathbb{Z}_2\sigma_2$ and $\mathbb{Z}_2\left(\begin{smallmatrix}
-1&0\\0&1\end{smallmatrix}\right)$ will only change the sign of the
coordinates $(x,y,z)$. Hence from \eqref{Autofkappa} we conclude
that there must be an element $\gamma \in \Gamma_{\Sigma}$ such that
$\gamma \cdot (x,y,z)\in(-2,2)\times\mathbb{R}\times\mathbb{R}$

 \hfill $\square$\end{proof}

\section{Infinitesimal transitivity}
\label{C:infinitesimaltransitivity} We consider  the abelian group
\begin{equation}\label{A_0}A_0=\{a:G\to G:[a(g),g]=1,\, \forall g \, \in \Ell, \, a\vert_{G\setminus \Ell}=1\}\end{equation} with group operation pointwise multiplication. We assume that the maps in $A_0$ are smooth, unless
noted otherwise, and we define $A$ to be the abelian group generated
by $A_0$ and the maps $a(g)=g^n$ for $n\in \Bbb Z$.  We refer to
\begin{equation} \label{frak_a}\frak a=\{x:G\to \mathfrak g:supp(x)\subset \Ell, \, Ad_g(x(g))=x(g),\, \forall g \in G\}\end{equation}
as the Lie algebra of $A$, because it has the crucial property
\begin{equation}\label{inclusion}exp(\mathfrak a)\subset A_0\subset A.\end{equation}

The group $A$ acts on $G\times G$ in two ways, corresponding to the
actions \eqref{twists}, by
\begin{equation} \label{groupAactions} A_1(a):(g,h)\to (ga(h)^{-1},h),\quad A_2(a):(g,h)\to (g,ha(g)^{-1}).\end{equation}
Note that the $T^n$ correspond to $a(g)=g^n$.

\begin{lem}\label{lemmaT_j,A_j} The Haar measurable function
$F(g,h)$ is $T_j$-invariant if and only if $F$ is $A_j$-invariant,
for $j=1,2$ (Here we can require the maps in $A$ to be $C^{\infty}$,
$C^0$, or merely measurable - the basic result is insensitive to
this requirement).
\end{lem}
Before beginning the proof, we recall a construction in \S 2 of
\cite{moore}.  Given a $\sigma$-finite measure space $(X,\mathcal
B,\mu )$ and a separable metric space $M$, Moore defines $U(X,M)$ to
consist of equivalence classes of $\mu$-measurable functions from
$X$ to $M$, where two functions are equivalent if they are equal
almost everywhere. This space is equipped with the topology of
convergence in measure with respect to a finite measure that
represents the measure class (\cite{moore} Proposition 6) and
depends only upon the measure class of $\mu$ and the topology of
$M$.  This kind of space is useful for us, because the
non-compactness of $G$ seems to preclude the use of $L^2$
techniques, and the natural action on $L^{\infty}$ is not
continuous.  If $M=\Bbb R$, we will simply write $U(X)$. By theorem
1 of \cite{moore} there is an isomorphism  \begin{equation}
\label{U_isomorphism} U(X\times Y)\cong U
\bigl(X;U(Y)\bigl).\end{equation}  We begin the proof of the lemma
using the isomorphism \ref{U_isomorphism}.

\begin{proof}  Let $X=G$ and $Y=G$, we have \begin{equation} \label{isomorphism_classes_functions} U(G\times G)\cong U \bigl(G;U(G)\bigl) \end{equation}
where $F(g,h)$ corresponds to the function of one variable $\bar
{F}:g\to F(g,\cdot )$.  Given  a measurable function $F(g,h)$,
$T_2$ acts on $F$ as follows: $F(T_2(g,h))=F(g,h g^{-1})$. Under the
above identification, we can think of $T_2$ as an operator that acts
on the corresponding $\bar{F}$.
$$\begin{array}{lccccl}
F(g,h) & \longrightarrow & \bar{F}: &g&\rightarrow & F(g,\cdot)\\
\downarrow T_2&&&&&\downarrow T_2\\
F(g,hg^{-1}) & \longrightarrow & \bar{F}: &g&\rightarrow & F(g,\cdot
g^{-1})
\end{array}$$
In other words, we can think of $T_2$ as the multiplication
operator that acts on a function $f$ of one variable by right
translation by $g$:
$$T_2:\bar{F}(g)\to R_g \bar{F} (g),$$  where is $R_g$ is  an operator that acts on a function of one variable by multiplication on the right: $ (R_gf)\vert_h=f\vert_{hg^{-1}}$.

 If a function $F$ is $A_1$ or $A_2$ invariant, by choosing $a : G \rightarrow G$: $a(g)=g$ we get directly that the function $F$ is  $T_1$ or $T_2$ invariant respectively.

Now  suppose that  a function $F$ is $T_2$ invariant i.e  $T_2 F=F$.
More explicitly $F(g,hg^{-1})=F(g,h)$ and  in terms of the
isomorphism \eqref{isomorphism_classes_functions} and the  notation
introduced above, $R_g\bar{F}(g)=\bar{F}(g)$ for almost every  $g$.
(This equality is true only for almost every $g$   since $U(X)$
consists of equivalent classes of functions). Then
$R_{g^n}\bar{F}(g)=\bar{F}(g)$ for all $n \in \mathbb{Z}$. Since the
action of $G$ on $U(G;U(G))$ is continuous (Proposition 12 of
\cite{moore}),  $R_a\bar{F}(g)=\bar{F}(g)$ for all $a$ in the
closure of the group generated by $g$.   Thus if $g$ is elliptic and
non-torsion, then $R_h\bar{F}(g)=\bar{F}(g)$ for all $h$ commuting
with $g$. The set of nontorsion elliptic elements $g$ has full
measure in $\Ell$, so the set of elements $h \in G$  for which
$R_h\bar{F}(g)=\bar{F}(g)$ does not hold  has measure zero.
Therefore we can conclude that $R_{a(g)}\bar{F}(g)=\bar{F}(g)$ for
almost every $g$ provided that $a \in A$. Going back to the initial
notation, this equation means that $F(g,ha(g)^{-1})=F(g,h)$. This
implies that $F$ is $A_2$-invariant. Similarly  if $F$ is $T_1$
invariant then $F$ is  $A_1$ invariant.    \hfill
$\square$\end{proof}

In general, given a Lie group $K$, we can always use left
translation to trivialize the tangent bundle:
$$ K \times \mathfrak{k}  \longrightarrow TK : (g, X)\rightarrow  \left. X\right|_g.$$ Suppose $X : K \rightarrow \mathfrak{k}$ and  $g \in K$.  Then $X(g)$ is a left invariant vector field. Note that $\left.X(g)\right|_g = \left.\dfrac{d}{dt}\right|_{t=0} g e^{t X(g)}$.  We can   identify $\Omega^0(K ; \mathfrak{k})$, the  set of functions from $K$ to $\mathfrak{k}$ with the space of  sections of $TK$, the set of vector fields on $K$:
$$\Omega^0(K;\mathfrak{k}) \leftrightarrow Vect(K) \, : \, X(\cdot) \leftrightarrow V_X,$$
where $\left.V_X\right|_g=\left.X(g)\right|_g\, \in
\left.TG\right|_g$.

Given $X$, $Y$ $: K \rightarrow \mathfrak{k}$, we associate  the
vector fields $V_X$ and $V_Y$ as in the discussion above  and we
form their commutator $[V_X, V_Y]=V_Z$. Then
\begin{equation}\label{commutator} Z(g) =\left.dY\right|_g
\left.\bigl(X(g)\right|_g \bigl)-\left.dX \right|_g
\left.\bigl(Y(g)\right|_g \bigl) +[X(g),Y(g)]\end{equation}

Let $\mathcal{G}$ denote the closure of the group of transformations
of $G\times G$ generated by $A_1$ and $A_2$. Recall
\begin{equation}  \label{A1A2actions} A_1(a):(g,h)\to \bigl(ga(h)^{-1},h\bigl),\quad A_2(a):(g,h)\to \bigl(g,ha(g)^{-1}\bigl).\end{equation}
We consider a continuous curve $ A_2(a_t)(g,h) \, \in \mathcal{G}$,  we differentiate at $(g,h)$ and we translate back at the identity:\begin{eqnarray}\left.\dfrac{d}{dt}\right|_{t=0} A_2(a_t)(g,h)&=&\left.\dfrac{d}{dt}\right|_{t=0}(g^{-1},h^{-1})\bigl(g,ha_t(g)^{-1}\bigl) \nonumber\\
&=&\left.\bigl(0,-h^{-1} ha_t^{-1}(g) a'_t(g)a_t^{-1}(g)\bigl)\right|_{t=0}\nonumber \\
&=&\bigl(0,- x(g)\bigl)\nonumber
\end{eqnarray}
So  the Lie algebra actions corresponding to \eqref{A1A2actions} are
given by
$$\begin{array}{ccccccccc}
x_1:&\mathfrak a&\rightarrow& vect(G\times G)&\text{  and  }& x_2:&\mathfrak a&\rightarrow& vect(G\times G)\\
&x&\mapsto &(-x(h),0)&&&x&\mapsto &(0,-x(g)).
\end{array}$$
These actions do not necessarily commute.

\begin{defin}

(a) $\mathfrak{G_0}$ is the Lie algebra  generated by the vector
fields on $G\times G$ given by
$$\{\bigl(y(h),x(g)\bigl) : x, y, \in \mathfrak{a}\}.$$

(b) $\mathfrak{G}$ is the Lie algebra of vector fields on $G \times
G$ generated by the family of Lie algebras $$\{ Ad_{\sigma}
\mathfrak{G_0} : \sigma \in A_1 \text{ or }  A_2\}.$$
\end{defin}
\begin{lem}\label{lemma:com_of_vf}
The bracket  in $\mathfrak{G_0}$ is given by \small\begin{equation}
[(x_1,y_1),(x_2,y_2)]|_{(g,h)}=\bigl(
dx_2|_h(y_1(g)|_h)-dx_1|_h(y_2(g)|_h), dy_2|_g
(x_1(h)|_g)-dy_1|_g(x_2(h)|_g)\bigl)\end{equation}\normalsize
\end{lem}

\begin{proof} Let  $X(g,h)=(x_1(h), y_1(g))$ and $Y(g,h)=(x_2(h),y_2(g))$, then
$$Z(g,h)=\left.X(g,h)\right|_{(g,h)}(\left.Y\right|_{(g,h)})-\left.Y(g,h)\right|_{(g,h)}(\left.X\right|_{(g,h)}) +[X(g,h),Y(g,h)]$$ as in \eqref{commutator}.
\begin{eqnarray}
&=&\left. \dfrac{d}{dt} \right|_{t=0} Y\left(g e^{tx_1(h)}, h e^{ty_1(g)}\right) - \left. \dfrac{d}{dt} \right|_{t=0} X\left(g e^{tx_2(h)}, h e^{ty_2(g)}\right)\nonumber\\
&=&\left. \dfrac{d}{dt} \right|_{t=0}\left(x_2(he^{ty_1(g)}), y_2(ge^{tx_1(h)})\right)-\left. \dfrac{d}{dt} \right|_{t=0}\left(x_1(he^{ty_2(g)}), y_1(ge^{tx_2(h)})\right) \nonumber\\
&=&\left.(dx_2\right|_h\left.(y_1(g)\right|_h)-\left.dx_1\right|_h\left.(y_2(g)\right|_h),
\left.dy_2\right|_g\left.(x_1(h)\right|_g)-\left.dy_1\right|_g\left.(x_2(h)\right|_g
)\nonumber
\end{eqnarray}
Notice that on the above calculations  the pointwise commutator
$$\left[\bigl(x_1(h),y_1(g)\bigl),
\bigl(x_2(h),y_2(g)\bigl)\right]=\bigl([x_1(h),x_2(h)], [y_1(g),
y_2(g)] \bigl)$$ vanishes. Since  $h$ and $g$ are regular elements,
they have distinct eigenvalues and  we can think of $h$ as a
diagonal matrix.  Since   $x_1(h)$ and $x_2(h)$ in $\mathfrak{g}$
commute with $h$, they both need to be in a diagonal form. Hence the
commutator $[x_1(h),x_2(h)]=0$. Similarly we have
$[y_1(g),y_2(g)]=0$.\hfill$\square$\end{proof}
\begin{lem}
Assuming we require maps to be $C^{\infty}$, we have $exp(\mathfrak
G)\subset \mathcal {G}$.
\end{lem}

\begin{proof} See Lemma (2.1.20) in \cite{PX}. \hfill $\square$\end{proof}

 Our goal now is to show that the Lie algebra $\mathfrak G$ is
infinitesimally transitive along certain fibers of the commutator
map $p$. We first calculate the derivative of the map $p$.
\begin{eqnarray} dp\vert_{(g,h)}:\mathfrak g\oplus \mathfrak g \to \mathfrak g: (\xi ,\eta) &\to& \xi^{
hgh^{-1}}-\xi^{hg}+\eta^{hg}-\eta^h \nonumber\\
 &=&(\xi^{h^{-1}}-\xi +\eta -\eta^{g^{-1}})^{hg}. \nonumber\end{eqnarray}
 To see this, we consider a curve on the tangent space of $G \times G$ at $(g,h)$ :  $\{(g_t,h_t) : t \in \mathbb{R}\}$, we differentiate it and then we translate at the identity: (at $t=0$, $(g_0,h_0)=(g,h))$.

\begin{eqnarray} &&\left.\dfrac{d}{dt}\right|_{t=0} [g,h]^{-1}[g_t,h_t]=\left.\dfrac{d}{dt}\right|_{t=0} hgh^{-1}g^{-1}g_th_tg_t^{-1}h_t^{-1} \nonumber\\
&=&hgh^{-1}g^{-1}(\left.\dfrac{d}{dt}\right|_{t=0} g_t)h_0g_0^{-1}h_0^{-1}+hgh^{-1}g^{-1}g_0(\left.\dfrac{d}{dt}\right|_{t=0} h_t) g_0^{-1}h_0^{-1}\nonumber\\
&-&hgh^{-1}g^{-1}g_0h_0(\left.\dfrac{d}{dt}\right|_{t=0} g_t^{-1}) g_0^{-1}-hgh^{-1}g^{-1}g_0h_0g_0^{-1} (\left.\dfrac{d}{dt}\right|_{t=0} h_t^{-1})\nonumber\\
&=&hgh^{-1} \xi h g^{-1} h^{-1}+hg \eta g^{-1} h^{-1}-h g \xi g^{-1} h^{-1}-h \eta h^{-1} \nonumber\\
&=& \xi^{hgh^{-1}}-\xi^{hg}+\eta^{hg}-\eta^{h}.\nonumber
\end{eqnarray}

So, $$ \begin{array}{lccccccl} p&:& G\times G &\to& G : &(g,h) &\to& [g,h]\\
dp\vert_{(g,h)}&:& \mathfrak{g} \oplus \mathfrak{g}&\to
&\mathfrak{g}
:&(\xi,\eta)&\to&\xi^{hgh^{-1}}-\xi^{hg}+\eta^{hg}-\eta^{h}.
\end{array} $$

\begin{lem} Suppose $g \in G$, then  the subalgebra  $\mathfrak {g}^g$ is either one dimensional  or is equal to   $\mathfrak {g}$.
\end{lem}
\begin{proof} We are going to consider different cases for $g$.
Suppose $g$ is elliptic. We identify $G$ with $\PSU(1,1)$ and we can
choose a basis to diagonalize $g$. So we can suppose that $g =
\left(\begin{smallmatrix}\lambda & 0\\0&  \lambda^{-1}
\end{smallmatrix}\right)$, with $|\lambda|=1$. If $\lambda = \pm 1$
then  $\mathfrak {g} ^g = \mathfrak {g}$. Otherwise  $\lambda \neq
\lambda^{-1}$ and in this case  all the elements in $\mathfrak{g}
\cong \mathfrak{su}(1,1)$  that commute with $g$  have the diagonal
form  $\left(\begin{smallmatrix} i x & 0\\0&  -i x
\end{smallmatrix}\right)$ where $x\, \in \mathbb{R}$. Therefore
$\mathfrak{g}^g = \left\{\left(\begin{smallmatrix} i x & 0\\0&  -i x
\end{smallmatrix}\right)  : x \in \mathbb{R} \right\}$ which is
one-dimensional.

Suppose $g$ is hyperbolic. We identify $G$ with $\PSL(2,
\mathbb{R})$ and we can choose a basis to diagonalize $g$. So we can
suppose that $g=\left(\begin{smallmatrix} \lambda & 0\\0&
\lambda^{-1} \end{smallmatrix}\right)$  where $\lambda >1$.   If
$\lambda \neq 1$, then all the elements in $\mathfrak{g} \cong
\mathfrak{sl}(2, \mathbb{R})$ that commute with $g$  have the
diagonal form  $\left(\begin{smallmatrix}  x & 0\\ 0 & -x
\end{smallmatrix}\right)$ where $x$ is in $\mathbb{R}$. Therefore
$\mathfrak{g}^g = \left\{\left(\begin{smallmatrix} x & 0\\0&  -x
\end{smallmatrix}\right)  : x \in \mathbb{R} \right\}$ which is
one-dimensional.

Finally, suppose $g$ is parabolic. We identify $G$ with $\PSL(2,
\mathbb{R})$ and we can choose a basis to write $g$ in the form
$\left(\begin{smallmatrix}  1 & 1\\ 0 & 1\end{smallmatrix}\right)$
if $g$ has eigenvalue equal to 1,  or in the form
$\left(\begin{smallmatrix}  1 & -1\\ 0 & 1\end{smallmatrix}\right)$
if it has eigenvalue -1. Then  an element in $\mathfrak{g}$ would
commute with $g$ if it  is of the form $\left(\begin{smallmatrix}  1
& x\\ 0 & 1\end{smallmatrix}\right)$ where $x$ is in $\mathbb{R}$.
Therefore $\mathfrak{g}^g = \left\{\left(\begin{smallmatrix} 1 &
x\\0&  1 \end{smallmatrix}\right)  : x \in \mathbb{R} \right\}$
which is one-dimensional. \hfill $\square$\end{proof}

\begin{lem} A point $(g,h)$ is regular for $p$ if and only if $\mathfrak{g}^g \cap \mathfrak{g}^h = \{0\}$.
\end{lem}

\begin{proof}
Suppose that the point $(g,h)$ is regular for $p$. This means that
$dp\vert_{(g,h)}$ is surjective i.e. $\Imaginary (dp|_{(g,h)}) =
\mathfrak{g}$. The Lie algebra $\mathfrak {g}$ has dimension three
hence $\Kernel (\left. dp\right|_{(g,h)})$  has dimension three as
well. The vector spaces  $\mathfrak{g}^h$ and  $\mathfrak{g}^g$ are
one dimensional (by the lemma above), so if their intersection
$\mathfrak{g}^g \cap \mathfrak{g}^h$ is non-empty, it  has
dimension one and in this case  we would have
\begin{equation}\label{eqvectorspaces}\mathfrak{g}^h \cap
\mathfrak{g}^g = \mathfrak{g}^h = \mathfrak{g}^g.\end{equation} We
prove that this will lead to a contradiction.

We can write the image of the map $dp|_{(g,h)}$  in terms of the
adjoint map as follows: \begin{equation}\label{ImageAd}\Imaginary\,
dp=\Imaginary (Ad\,h^{-1}-1)+\Imaginary
(Ad\,g^{-1}-1).\end{equation} Also, we can think of $\mathfrak{g}^h$
and $\mathfrak{g}^g$ as  $$\mathfrak{g}^h=\Kernel\, (Ad \, h^{-1}-1)
\text{ and  } \mathfrak{g}^g=\Kernel\, (Ad \, g^{-1}-1).$$ From
\eqref{eqvectorspaces}, we can obtain that $\Imaginary
(Ad\,h^{-1}-1)=\Imaginary (Ad\,g^{-1}-1)$. The space $\Imaginary
(Ad\,h^{-1}-1)$ has dimension two but the image of the map $dp$  has
dimension three since $dp$ is surjective. This contradicts
\eqref{ImageAd} and hence the   assumption   $\mathfrak{g}^g \cap
\mathfrak{g}^h \neq  \{0\}$.

For the other direction  we assume that
 $\mathfrak{g}^g \cap \mathfrak{g}^h = \{0\}$ . To prove that $p$ is regular for $(g,h)$ we need to prove that $dp|_{(g,h)}$ is surjective. Therefore it would be enough to show that $\Imaginary (dp)$ has dimension three. Since $\Imaginary (dp)$ has two subspaces that have dimension two, it has to have at least dimension two. If it  had dimension two, then the two subspaces $ \Imaginary (Ad\,h^{-1}-1)$ and  $\Imaginary (Ad\,g^{-1}-1)$ need to be equal. But this would imply that $\mathfrak{g}^g = \mathfrak{g}^h$ which contradicts our assumption. \hfill $\square$ \end{proof}

\begin{lem} \label{regularpoints}A point $(g,h)$  is regular for $p$ if and only if $[g,h] \neq 1$
\end{lem}

\begin{proof} Using the above lemma, it is equivalent to show that $$ \mathfrak{g}^g \cap \mathfrak{g}^h = \{0\} \text{  if and only if } [g,h] \neq 1,$$
$$ \text{ or } \mathfrak{g}^g \cap \mathfrak{g}^h  \neq \{0\} \text{  if and only if } [g,h] = 1,$$
Suppose there exists a non-zero $x  \, \in \mathfrak{g}^g \cap
\mathfrak{g}^h$. Then we have $gxg^{-1}=x$ and $hxh^{-1}=x$.

Since $x \in \mathfrak{sl} (2, \mathbb{R})$, it has the form $x=
\bigl(\begin{smallmatrix} a & b \\ c &-a \end{smallmatrix}\bigl)$.
The characteristic equation for this matrix is $$\lambda^2+\det
x=0.$$

 If  $\det x <0$, $x$ has two distinct real eigenvalues, therefore we can choose a basis  so that $x$ would have diagonal form. This would force $g$ and $h$ to be diagonal as well, and therefore they would have to commute.

If $\det x > 0$, $x$ has two purely imaginary (distinct)
eigenvalues, so we can use the same argument as above, and obtain
that $g$ and $h$ commute.

Finally, if  $\det x =0$, $\lambda =0$ and  we can choose a basis
so that we can write  $x= \bigl(\begin{smallmatrix} 0 & 1 \\ 0 &0
\end{smallmatrix}\bigl)$. Then   $g$ and $h$ need to have the form
$x= \bigl(\begin{smallmatrix} 1 & a \\ 0 & 1
\end{smallmatrix}\bigl)$, for $a \in \mathbb{R}$. Notice that
matrices of this form  commute with each other.

Suppose now  that $[g,h]=1$. We want to prove that $\mathfrak{g}^g
\cap \mathfrak{g}^h  \neq \{0\}$. We are going to consider different
cases: If $g$ is hyperbolic, we can pick a basis so that
$g=\bigl(\begin{smallmatrix} \lambda& 0 \\ 0 & \frac{1}{\lambda}
\end{smallmatrix}\bigl)$. Since  $h$ commutes with $g$ it should
also be diagonal.  If we choose $x \neq 0 \, \in \mathfrak{g} $ in
diagonal form, it is obvious that $x$ would commute with both $g$
and $h$.  This would imply that $x \in \mathfrak{g}^g \cap
\mathfrak{g}^h$.

If $g$ is elliptic, we can pick a basis so that
$g=\bigl(\begin{smallmatrix} \alpha& 0 \\ 0 & \bar{\alpha}
\end{smallmatrix}\bigl)$, and using the  same argument as above we
can find a non-zero element $x \in \mathfrak{g}$ such that $x \in
\mathfrak{g}^h \cap \mathfrak{g}^g$.

If $g$ is parabolic, we can choose a basis so that $g=
\bigl(\begin{smallmatrix} 1 & 1 \\ 0 &1 \end{smallmatrix}\bigl)$.
The condition that $h$ commutes with $g$  forces $h$ to have the
form  $h= \bigl(\begin{smallmatrix} 1 & b \\ 0 &1
\end{smallmatrix}\bigl)$. If we pick $x =  \bigl(\begin{smallmatrix}
0 & 1 \\ 0 &0 \end{smallmatrix}\bigl) \in \mathfrak{g}$, we see that
$x$ commutes with both $g$ and $h$ and hence in all cases $\mathfrak
g ^g \cap \mathfrak g^h \neq \{ 0\}$

\hfill $\square$\end{proof}
\begin{prop}\label{inf_transitivity}
For $(g,h) \in \Ell \times \Ell$ such that $[g,h] \neq 1$, the
evaluation map
$$\left.eval\right|_{(g,h)} : \mathfrak{G} \to \Kernel \left(\left.dp\right|_{(g,h)}\right)$$
 is surjective.
\end{prop}
\begin{proof} Let $g,h\in \Ell$  such that $[g,h] \neq 1$.
 Recall that $\mathfrak G_0$ is the Lie algebra that
 consists of vector fields on $G \times G$ that are sums of vector fields of the form $\bigl( x(h),0\bigl), \, \bigl(0,y(g)\bigl)$, where $x, \ y$ are
 elements in $\mathfrak a$, and recall that $\mathfrak a$ consists
 of maps $x : G \rightarrow \mathfrak{g}$ that have the property $Ad_g (x(g))=x(g)$
 for every $g \, \in G$. This means that $\mathfrak g^h\oplus \mathfrak g^g \subset eval\vert_{(g,h)}(\mathfrak G_0) $. In addition,  when $\xi$ commutes with $h$ and $\eta$ commutes with $g$, $dp|_{(g,h)}=0$ and therefore $(\xi, \eta)$ belongs in the  $\Kernel (dp|_{(g,h)})$. Hence we have  the following
\begin{equation} \mathfrak g^h\oplus
\mathfrak g^g\subset \Kernel (dp\vert_{(g,h)})\subset \mathfrak
g\oplus \mathfrak g.
\end{equation}

Furthermore $(g,h)$ is regular for $p$ if and only if $[g,h]\neq 1$
by Lemma \ref{regularpoints}.  Thus, at a point $(g,h) \in \Ell
\times \Ell$  the Lie algebras $\mathfrak {g} ^h$ and $\mathfrak {g}
^g$ have dimension 1 respectively.  Therefore, by showing that the
quotient
\begin{equation}\label{quotient}eval\vert_{g,h}(\frak G)/(\frak
g^h\oplus \frak g^g)\end{equation} is not zero, we  show that
$eval|_{(g,h)}\mathfrak {G}$ has dimension 3, hence it is equal to
the $\Kernel(dp|_{(g,h)})$ and this will suffice to show that the
evaluation map is surjective.

We are going to find a non-zero element of $eval|_{(g,h)}\mathfrak
{G}$, that does not belong to $\mathfrak g^h\oplus \mathfrak g^g $.
Let $\bigl(x_1(h), y_1(g)\bigl)$,  $\bigl(x_2(h), y_2(g)\bigl)$ two
vector fields in $\mathfrak {G}$. Their commutator is given by the
following formula (Lemma \ref{lemma:com_of_vf}) \small $$
[(x_1,y_1),(x_2,y_2)]|_{(g,h)}=\bigl(
dx_2|_h(y_1(g)|_h)-dx_1|_h(y_2(g)|_h), dy_2|_g
(x_1(h)|_g)-dy_2|_g(x_2(h)|_g)\bigl)$$\normalsize We are going to
show that the commutator of two such  vector fields does not belong
to $\mathfrak g^h\oplus \mathfrak g^g $.  We could rewrite  the
first component as $$\left.\dfrac{d}{dt}\right|_{t=0} x_2(h
e^{ty_1(g)})-x_1(he^{ty_2(g)})$$From the way that the vector fields
$x_1$ and $x_2$ are defined, $h e^{ty_1(g)}$ has to commute with
$x_2(h e^{ty_1(g)})$ and  $h e^{ty_2(g)}$ has to commute with $x_1(h
e^{ty_2(g)})$. Explicitly,
\small$$\left.\dfrac{d}{dt}\right|_{t=0}Ad(h
e^{ty_1(g)})\biggl(x_2(he^{ty_1(g)})\biggl) -
\left.\dfrac{d}{dt}\right|_{t=0}Ad(h
e^{ty_2(g)})\biggl(x_1(he^{ty_2(g)})\biggl) $$ $$=
\left.\dfrac{d}{dt}\right|_{t=0}\biggl( x_2(h
e^{ty_1(g)})-x_1(he^{ty_2(g)})\biggl)$$\normalsize Calculating the
derivatives on both sides, we get \small
$$Ad(h)\biggl(\left.y_1(g)\right|_h(x_2) +[y_1(g),x_2(h)]\biggl) -
Ad(h)\biggl(\left.y_2(g)\right|_h(x_1) +[y_2(g),x_1(h)]\biggl)$$ $$
=\left.y_1(g)\right|_h(x_2)-\left.y_2(g)\right|_h(x_1)$$\normalsize
We notice that if the term
$$Ad(h)\biggl([y_1(g),x_2(h)]-[y_2(g),x_1(h)]\biggl)$$is not zero,
the expression $\bigl(\left.
dx_2\right|_h(y_1(g)|_h)-dx_1|_h(y_2(g)|_h)\bigl)$ does not belong
in $\mathfrak{g}^h$. To achieve this we can  choose $y_2=0,  \,
x_1=0 $, and since $g$ and $h$ do not commute, we can find $y_1$ and
$x_2$ so that the commutator $[y_1(g),x_2(h)]$ is not
trivial.\hfill$\square$\end{proof}

\section{Proof of Theorem \ref{actionofMCG}}\label{C:proofoftheorem}



Let $\tilde{g_c} \in \tilde G$ that covers $g_c' \in
\SL(2,\mathbb{R})
$ and $t=\tr g_c'$. There are 3 cases to consider:

\begin{itemize} \item Suppose $t<2$. From Goldman's theorem in \cite{goldman03}, the action of $\Aut(\kappa)$ on $\kappa^{-1}(t)$ is properly discontinuous and from lemma \ref{properaction}, we obtain that the action of $\Gamma_{\Sigma}$ on  $R_1^{-1}(\tilde{g_c})$ is properly discontinuous.
\item Suppose $2<t<18$. From theorem  5.2.1 in \cite{goldman03}  and lemma \ref{lemma2<t<18},   given any pair $(g,h)\in \SL(2,\mathbb{R})\times \SL(2,\mathbb{R})$, with $[g,h]=g_c'$, we can apply an element of the mapping class group  to change it  to $(g',h')$, where $g'$ is an elliptic element. The pair $(g',h')$ projects to a pair in $G\times G$, which we will also denote by $(g',h') \in \Ell \times G $.

To use the infinitesimal transitivity  method in section
\ref{C:infinitesimaltransitivity}, we need to make also  $h'$
elliptic. To do this, we can apply $A_2 \in \mathcal{G}$ to the pair
$(g',h')$, $$A_2^{-n}(a):\ (g',h') \longrightarrow (g',h'
a(g')^n).$$ We choose a basis so that $g'$ is a rotation. Then we
can find a map $a \in A_0$  so that $a(g')$ will be a rotation of
infinite order   and for sufficiently large $n \in \mathbb{N}$, $h'
a(g')^n$ becomes elliptic. To see this we can suppose that
$G=\PSU(1,1)$,  $a(g')=\left(\begin{smallmatrix}
\lambda&0\\0&\bar{\lambda} \end{smallmatrix}\right)$ and
$h'=\left(\begin{smallmatrix} \alpha&\beta\\\bar{\beta}&\bar{\alpha}
\end{smallmatrix}\right)$. We apply  $A_2^{-1}$  $n$ times:
$$A_2^{-n}(a) :(g',h') \longrightarrow (g', h'a(g')^n).$$ Then $\tr
(h'a(g')^n)= 2 \text{Re}\bigl(\alpha \lambda^n)$, and so we can
choose $n$ so that the $|\tr (h'a(g')^n)|$ is as small as we wish
since $\{\lambda ^n\}$ will be dense in the circle.

 Hence, under the action of $\Gamma_{\Sigma}$ and $\mathcal{G}$ the pair $(g,h)$ can be transformed to a pair of elliptic elements.  In addition, since  $t>2$ the elements  $g$ and $h$ do not commute, therefore the pair $(g,h)$ is regular and  from proposition \ref{inf_transitivity}  we conclude that $\mathcal{G}$ is infinitesimally transitive along the fiber $R_1^{-1}(\tilde{g_c})$.

Suppose $F$ is a Haar measurable function defined on $G\times G$
that  is $\Gamma_{\Sigma}$ invariant. By lemma \ref{lemmaT_j,A_j} it
is going to be $\mathcal{G}$ invariant. Given any point $(g,h)$ in
$G\times G$ the $\mathcal{G}$-orbit of $(g,h)$ is the whole fiber
and  $\mathcal{G}$ is infinitesimally transitive along the fiber,
thus an invariant $F$ is  constant on the fibers for almost every
$\tilde {g_c}$.

\item Suppose $t>18$.  In Goldman's setting the action of $\Aut(\kappa)$ separates $\kappa^{-1}(t)\cap \mathbb{R}^3$ in to  two regions (section 5 in \cite{goldman03}). Let $\Omega=\Aut(\kappa)\cdot \bigl(\Omega_0 \cap \kappa^{-1}(t)\bigl) \subset \kappa^{-1}(t)$, where $$\Omega_0=(-\infty,-2)\times (-\infty,-2)\times(-\infty,-2)$$ The action of $\Aut(\kappa)$ on $\Omega$   is properly discontinuous and on the complement of $\Omega$ the action is ergodic. Define the set $$\Omega'=\{(g,h) \in \SL(2,\mathbb{R}) \times \SL(2,\mathbb{R}) :  \text{ such that } [g,h]=  g_c'\}, $$ where $ \tr g_c'= t \text{ and } (\tr g, \tr h,\tr gh)\in \Omega$.

Then by lemma \ref{properaction} the action of $\Gamma_{\Sigma}$ on
$\Omega'$ is properly discontinuous.  Let $(g,h) \not\in \Omega'$
then by theorem 5.2.1 in \cite{goldman03}  and  lemma
\ref{lemma2<t<18}, we can find $\gamma \in \Gamma_{\Sigma}$ such
that $\gamma \cdot (g,h)$ is a pair of elliptic elements. These
pairs consist a set where infinitesimal transitivity holds, and by
similar argument as in the  case where $2<t<18$,  we can prove
ergodicity. \end{itemize}

%

\bibliography{bibliography}

\begin{thebibliography}{MKS70}

\bibitem[Gol]{Goldman06}
W.~Goldman.
\newblock Mapping class group dynamics on surface group representations.
\newblock to appear in {"}Problems on Mapping Class Groups and Related Topics
  {"} B. Farb, ed. Proc. of Symposia in Pure Math. Amer. Math. Soc.,
  http://www.math.umd.edu/{$\sim$}wmg/publications.html.

\bibitem[Gol03]{goldman03}
W.~Goldman.
\newblock The modular group action on real {SL}(2)-characters of a one-holed
  torus.
\newblock {\em Geometry and Topology}, 7:443--486, 2003.

\bibitem[Kon06]{mydissertation}
P.~Konstantinou.
\newblock {\em Homomorphisms of the fundamental group of a surface into
  PSU(1,1) and the action of the mapping class group}.
\newblock PhD thesis, The University of Arizona, 2006.

\bibitem[LS77]{LS}
R.~Lyndon and R.~Schupp.
\newblock {\em Combinatorial {G}roup {T}heory}.
\newblock Springer-Verlag, Berlin, Heidelberg, New York, 1977.

\bibitem[MKS70]{MKS}
W.~Magnus, A.~Karrass, and D.~Solitar.
\newblock {\em Combinatorial {G}roup {T}heory: Presentations of groups in terms
  of generators and relations}.
\newblock Dover Publications, New York, 1970.

\bibitem[Moo76]{moore}
C.~Moore.
\newblock Group extensions and cohomology for locally compact groups. {III}.
\newblock {\em Trans. {A}mer. {M}ath. {S}oc.}, 221:1--33, 1976.

\bibitem[Nie64]{nielsen}
J.~Nielsen.
\newblock Die isomorphismen der allgmeinen unendlichen gruppe mit zwei
  {E}rzeugenden.
\newblock {\em {M}ath. {A}nn.}, 78:385--397, 1964.

\bibitem[PX02]{PX}
D.~Pickrell and E.~Z. Xia.
\newblock Ergodicity of mapping class group actions on representation
  varieties, {I}. {C}losed surfaces.
\newblock {\em Comment. Math. Helv.}, 77:339--362, 2002.

\bibitem[Ran69]{rankin}
R.~Rankin.
\newblock {\em The modular group and its subgroups}.
\newblock Madras Ramanujan Institute, 1969.

\end{thebibliography}

\end{document}